\documentclass{amsart}

\usepackage{amssymb}
\usepackage{amsmath}

\newtheorem{theorem}{Theorem}[section]
\newtheorem{lemma}[theorem]{Lemma}
\newtheorem{proposition}[theorem]{Proposition}
\newtheorem{corollary}[theorem]{Corollary}
\theoremstyle{definition}

\theoremstyle{remark}
\newtheorem{remark}[theorem]{Remark}

\numberwithin{equation}{section}

\newcommand{\R}{\mathbb{R}}
\newcommand{\N}{\mathbb{N}}

\renewcommand{\H}{\mathcal{H}}

\author{Mark Allen}
\address{Department of Mathematics, The University of Texas at Austin, Austin,
  TX 78712}
\email{mallen@math.utexas.edu}

\thanks{M.~Allen is supported by NSF grant DMS-1303632}

\title{A Fractional Free Boundary Problem related to a Plasma Problem}

\begin{document}

\begin{abstract}
 We study a fractional analogue of a plasma problem arising from physics. Specifically, for a fixed bounded domain
 $\Omega$ we study solutions to the eigenfunction equation
  \[
   (- \Delta)^s u = \lambda(u- \gamma)_+
  \]
 with $u \equiv 0$ on $\partial \Omega$.
\end{abstract}

\maketitle
\section{Introduction}    \label{s:introduction}
 A mathematical model for the region inhabiated by plasma in a Tokamak machine is given by the two dimensional equation
  \[
   \sum_{i=1}^2 \frac{\partial}{\partial x_i} \left( \frac{1}{x_1} \frac{\partial u}{\partial x_i}  \right) = u_+ 
  \]
 in a bounded domain $\Omega$. $u_+$ is the positive part of the function $u$. 
 The region inhabited by the plasma is given by $\{u>0\}$. Properties of solutions to
 this and similar problems was studied in \cite{t75} and \cite{bb80}. The simplified model with solutions to
  \begin{equation}  \label{e:ploc}
   -\Delta u = \lambda u_+
  \end{equation}
 was studied in \cite{t77}. Regularity of the free boundary $\partial \{u >0\}$ was studied in \cite{ks78}.  
 The physical applications of the simplified model \eqref{e:ploc} are for two dimensions; however, one may study
 \eqref{e:ploc} in higher dimensions. The regularity of the free boundary for higher dimensions was studied in 
 \cite{ks78}.  
 
 In this paper we study a fractional analogue of \eqref{e:ploc}. For a bounded domain $\Omega$ in $\R^n$ we consider
 solutions to the equation
  \begin{equation}  \label{e:p}
   (-\Delta)^s u = \lambda(u - \gamma)_+
  \end{equation}
 with $u \equiv 0$ on $\partial \Omega$. 
 The reason for the translation $(u-\gamma)$ is to utilize the definition of the fractional Laplacian with zero dirichlet data. 
 In Section \ref{s:fl} we use the extensiion operator and subtract out the constant $\gamma$ which gives an equation of the sort
  \[
   ``(-\Delta)^s u = \lambda u_+ \ ''
  \]
 The main aim of this paper is to study 
 properties of the free boundary $\{u(\cdot, 0)=0\}$.

 \subsection{Outline}
 The outline of the paper is as follows: In Section \ref{s:fl} we establish certain properties of the fractional Laplacian that will be needed in our paper. 
 We also discuss the notion of 
 the extension operator that allows one to ``localize'' the fractional Laplacian. 
 In Section  \ref{s:e} we prove existence of solutions to \eqref{e:p}.
 In Section  \ref{s:intreg} we prove interior regularity for solutions. 
 In Section  \ref{s:top} we begin the study of the free boundary. We prove topological properties of 
 the free boundary and show how they may differ from the original local plasma problem \eqref{e:ploc}. 
 In Section \ref{s:freereg} we use an Almgren's type frequency function to classify so-called blow-up solutions. The classification of blow-ups allows
 us a classification of the free boundary points. We then give a regularity result for the ``regular set'' of the free boundary. 
 In Section \ref{s:sing} we define the singular set and prove a Hausdorff dimensional bound for the singular set which 
 shows that the singular set is ``small''.
 
  \subsection{Notation}  \label{s:not} 
 The notation for this paper will be as follows. Throughout the paper $2s=1-a$ and $-1<a<1$ and $s$ will always 
 refer to the order of the fractional Laplacian $(-\Delta)^s$. $(x,y) \in R^{n+1}$ with
 $x \in \R^{n}$ and $y \in \R$. $\Omega$ will always be a smooth bounded domain. For a set $U \in \R^{n+1}$,
 
 - $L^2(a, U) := \{f \mid f|y|^{a/2} \in L^2(U)$\}.
 
 - $L^2(a, \partial U):= \{f \mid f|y|^{a/2} \in L^2(\partial U)\}$ with respect to $H^{n}$ Hausdorff measure.  
 
 - $H^1(a,U) := \{f \mid f, \nabla f \in L^2(a,U)$\}.
 
 - $U':= \{x \in \R^{n} \mid (x,0) \in U\}$. 
 
 - We refer to $\R^{n} \times \{0\}$ as the thin space.
 
 - $B_r := \{x \in \R^n \mid |x| <1 \}$

 - $f_{\pm}$ denote the positive and negative parts of $x$ respectively so that $f=f_+ - f_-$. 
 
 We denote the free boundary as $\Gamma =  \{u(\cdot,0) = 0\}$.

\section{Fractional Laplacian} \label{s:fl}
 We define the fractional Laplacian through the spectral decomposition. For a bounded domain $\Omega$ let 
 $0<\lambda_1<\lambda_2 \leq \ldots $ and $\{\phi\}$ be the eigenvalues and corresponding orthonormalized eigenfunctions
 with dirichlet zero boundary data. For $f \in L^2(\Omega)$ we write 
  \begin{equation}  \label{e:f}
   f = \sum_{k = 1}{a_k \phi_k}.
  \end{equation}
 Then the fractional Laplacian is given by
  \begin{equation}  \label{e:frac}
   (- \Delta)^s f(x) = \sum_{i=1}{\lambda_i^s a_i \phi_i(x)}. 
  \end{equation}
 The fractional Laplacian can also be given as a Dirichlet to Neumann boundary data map by the use 
 of an extension operator. In the case when $(- \Delta)^s$ is defined on all of $\R^n$ instead of on a bounded domain
 this equivalency was given in the paper \cite{cs07}. 
 For a bounded domain there is an analoguous extension operator \cite{st10}. We look at the solution to 
 the following weighted elliptic problem in an extra dimension 
  \[
   \begin{aligned}
    \text{div}(y^a \nabla u) &= 0 \text{ in } \Omega \times \R^+ \\
    u(x,0) &= f(x) \\
    u(x,y) &= 0 \text{  for  } (x,y) \in \partial \Omega \times \R^+\\
    u(x,y) & \to 0 \text{  as  } y \to \infty 
   \end{aligned}
  \]
 where $a=1-2s$. 
 The fractional Laplacian is a Dirichlet to Neumann boundary data map:
  \begin{equation}  \label{e:extend}
   c_a \lim_{y \to 0} y^{-a} u_y(x,y) = (- \Delta)^s u.  
  \end{equation}
 where $c_a$ is a constant only depending on $a$ and dimension $n$. 
 If we shift the solution downward by subtractin the constant $\gamma$ our equation of study then becomes
  \[
   c_a \lim_{y \to 0} y^{-a} u_y(x,y) = \lambda u_+ \text{  for  every  } x \in \Omega.
  \]
 Absorbing the constant into the right hand side we obtain 
  \begin{equation}  \label{e:eigen}
    \lim_{y \to 0} y^{-a} u_y(x,y) = -\lambda u_+ \text{  for  every  } x \in \Omega,
  \end{equation}
 where $\lambda>0$ is a new constant. Many of the results in this paper will apply to solutions of the more general equation \eqref{e:eigen}. However, if 
 $u$ is a solution to \eqref{e:p} then after subtracting $\gamma$, the extension function $u$ will satisfy \eqref{e:eigen}, but will also satisfy the additional condition
  \begin{equation}   \label{e:do}
   u_y \leq 0.
  \end{equation}
 Condition \eqref{e:do} will be used in Section \ref{s:sing} to give a Hausdorff dimensional bound on the singular set. 
 That solutions of \eqref{e:p} satisfy \eqref{e:do} is easily seen from the following argument. 
 If $u$ is the extension for a solution to \eqref{e:p}, then $y^a u_y$ is $-a$-harmonic (\cite{cs07}). Furthermore, $y^a u_y \to 0$ as $y \to \infty$. Then
 $y^a u_y$ has nonpositive 
 boundary data on $\partial(\Omega \times \R^+)$, and so from the maximum principle for $-a$-harmonic functions (\cite{FKJ}) we conclude that
 $u_y \leq 0$. 
 
  We will utilize the following notion of trace for the weight $y^a$ (see \cite{ALP}). 
 \begin{proposition}  \label{p:trace}
  Let $U \subset \R^{n+1}$ an open Lipschitz domain. Then there exists two compact operators
   \[
    \begin{aligned}
     T_1: H^1(a,U) &\hookrightarrow L^2(a, \partial U) \\
     T_2: H^1(a,U) &\hookrightarrow L^2(U')
    \end{aligned}
   \]
 \end{proposition}
 
 By utilizing rescaling and Proposition \ref{p:trace} on $B_1$ we obtain the following
  \begin{corollary}  \label{c:trace}
   Let $v \in H^1(a,B_r^+)$. Then there exists a constant $C =C(n,a)$ such that
    \[
     \begin{aligned}
      \int_{(\partial B_r)^+}{y^a u^2} d \H^n &\leq C r \int_{B_r^+}{y^a |\nabla u|^2} \\
      \int_{B_{r}'}{u^2}d \H^n  &\leq C r^{1-a} \int_{B_r^+}{y^a |\nabla u|^2}
     \end{aligned}
    \]
  \end{corollary}

 We will also need the following Hopf type Lemma
  \begin{lemma}  \label{l:hopf}
   Let $v$ be a non-constant $a$-harmonic in $B_r^+$ for some $r>0$. Assume $v$ achieves its minimimum at $(x_0,0) \in B_{r}'$. Then 
    \[
     \lim_{y \to 0} \frac{v(x_0,y)}{y^{1-a}} >0. 
    \]
  \end{lemma}

  \begin{proof}
   By subtracting a constant we may assume $v(x_0,0)=0$ and therefore $v \geq 0$.  Let $w$ 
   (not identically zero) be an $a$-harmonic function satisfying $w(x,0)=0$ and $w\leq v$ on $\partial B_r$. By the Boundary Harnack Principle for 
   $a$-harmonic functions stated in \cite{crs10} we have for $\rho < $dist$(x_0, \partial B_r)/2$
    \[
     \sup_{ B_{\rho}(x_0,0)} \frac{w}{y^{1-a}} \leq C(r) \inf_{ B_{\rho}(x_0,0)} \frac{w}{y^{1-a}}.
    \]
   Then 
    \[
     \inf_{ B_{\rho}(x_0,0)} \frac{w(x,y)}{y^{1-a}} >0,
    \]
   and so  
    \[
     \lim_{y \to 0}\frac{v(x_0,y)}{y^{1-a}}  \geq  \lim_{y \to 0} \frac{w(x_0,y)}{y^{1-a}} > 0. 
    \]
  \end{proof}
 
  \begin{remark}
   The Boundary Harnack Priniple is a powerful tool, but is not necessary to prove Lemma \ref{l:hopf}. One can for instance use the power series
   representation for odd $a$-harmonic functions shown in \cite{a13} to achieve the same result for $w$ in the above proof and hence for $v$. 
  \end{remark}

\section{Existence}  \label{s:e}
 In this section we prove existence of solutions to \eqref{e:p}. To obtain an eigenfunction we consider minimizing
 the fractional energy 
  \begin{equation}  \label{e:energy}
   D(u):=\int_{\Omega}{u(- \Delta)^s u}
  \end{equation}
 subject to the constraint  
  \begin{equation}  \label{e:constraint}
   G(u):=\int_{\Omega}{(u - \gamma)_+} = c
  \end{equation}   
 where $c, \gamma >0$ are two fixed constants. Using the extension mentioned in Section \ref{s:fl}, this is equivalent to minimizing
  \[
  \iint\limits_{\Omega \times \R^+} y^a |\nabla w|^2  \ dx \ dy ,
  \]
 subject to the constraint \eqref{e:constraint}.
 \begin{lemma}
  There exists a minimizer of \eqref{e:energy} subject to the constraint \eqref{e:constraint}.
 \end{lemma}
 
 \begin{proof}
  Minimizing \eqref{e:energy} is equivalent to minimizing
   \[
    \iint\limits_{\Omega\times \R^+} y^a |\nabla w|^2  \ dx \ dy
   \]
  with $w(x,0)=u(x)$ and $w(x,y) \to 0$ as $y \to \infty$. 
  From Corollary \ref{c:trace} it follows that $L^2(\Omega)$ is compactly contained in $H^1(a,\Omega \times \R^+)$. The existence of a minimizer then immediately follows.  
 \end{proof}
  We note that the extensions is not necessary to prove a compactness theorem. Using only the spectral decomposition there is an elementary proof using power series
  that if a sequence $u_k$, is bounded in $H^s(\Omega)$, then there exists $u_0$ and a subsequence such that $u_k \rightharpoonup u_0$ in $H^s(\Omega)$ 
  and $u_k \to u_0$ in 
  $L^2(\Omega)$. 
   
 \begin{lemma}
  For a domain $\Omega$ and fixed constant $c$ there exists a solution to \eqref{e:p} for some $\lambda>0$. 
 \end{lemma}
 
 \begin{proof}
  Since the functionals $D: H^s \to \R$ and $G:H^{s/2} \to \R$ have the Frechet derivatives 
   \[
    D'(u)= (- \Delta)^{s/2}u \qquad G'(u):= 1/2\int_{\Omega}{(u - \gamma)_+}
   \]
  then there exists $\lambda>0$ such that for a minimizer of $D(u)$ subject to the constraint $G(u)=c$ satisfies  
   \[
    D'(u) = \lambda G'(u).
   \]
  See \cite{z95} for a discussion of Lagrange Multipliers with Frechet derivatives. 
 \end{proof}
 
 The above Lemma shows existence for some $\lambda>0$. In \cite{t77} existence is shown for fixed $\lambda$ belonging to a correct range. 
 It would be of interest to show existence for fixed $\lambda$ for this fractional problem as well. Since this paper focuses on properties
 of the free boundary, we chose to only give a quick proof for some $\lambda$ to simply demonstrate that our class of solutions we study is nonempty.

 
 

\section{Interior Regularity of Solutions}  \label{s:intreg}
 In this section we obtain the interior regularity of solutions which will enable us to obtain regularity of
 the free boundary where the gradient does not vanish. Since we are dealing with an eigenvalue equation, 
 we use a bootstrap technique. Obtaining regularity for $u$ 
 passes the same regularity to $\lambda u^+$ up to Lipschitz regularity. 
 Regularity for $\lambda u^+ = ``(-\Delta)^s u ''$ allows us to obtain higher regularity for $u$. To use
 the bootstrap technique we utilize the following Proposition from \cite{S07}.
 
 \begin{proposition}   \label{p:holder}
  Let $w = (- \Delta)^s u$. Assume $w \in C^{0,\alpha}(\R^n)$ and $u \in L^{\infty}$ for $\alpha \in (0,1)$ and 
  $s > 0$. Then
   \[
    \begin{aligned}
     &\text{If } \alpha + 2s \leq 1 & & \text{  then  } u \in C^{0,\alpha+2s}(\R^n), \\
     &\text{If } \alpha + 2s >1     & & \text{  then  } u \in C^{1,\alpha+2s-1}(\R^n). 
    \end{aligned}
   \]
 \end{proposition} 
 
 Using Proposition \ref{p:holder} we are able to obtain the following interior regularity result. 
 \begin{theorem}  \label{t:solreg}
  Let $u$ be a solution to \eqref{e:eigen}. Then for any $K \Subset \Omega$
    \[
     \begin{aligned}
      &\text{If } s\leq 1/2,  & &\text{ then  }   u \in C^{1,\alpha}(K) \text{  for every  } \alpha<2s,\\
      &\text{If } s > 1/2,   & &\text{ then  }   u \in C^{2,\alpha}(K) \text{  for every  } \alpha<2s-1
     \end{aligned}
    \]
 \end{theorem}
 
 \begin{proof}
  Throughout the proof $K$ will be any compact subset of $\Omega$. Each time we obtain higher regularity in 
  the bootstrap technique we will pass to a smaller subdomain. By abuse of notation we will continue to call
  the smaller set $K$. The Holder norms will of course depend on $K$.
  To utilize Proposition \ref{p:holder}, which applies to the fractional Laplacian whose domain is functions
  defined on all of $\R^n$, we use the Reisz potential as well as the $a$-harmonic extension of $u$ in 
  $\R^n \times \R^+$.  
   \[
    u_0(z) := (-\Delta)_{\R^n}^{-s} \lambda u^+ = \int_{K}{\frac{\lambda u^+(y)}{|z-y|^{n-2s}} \ dy}. 
   \] 
  If we consider the variable $z \in \R^n \times \R^+$, then
  $u_0$ is $a$-harmonic in $\R^n \times \R^+$. Furthermore  
   \[
    \lim_{y \to 0} y^a \partial_y (u - u_0)(x,y) =0 \text{  for any } x \in K
   \]
  By even reflection in the $y$ variable $u - u_0$ is $a$-harmonic in the 
  cylinder $K \times \R$. $a$-harmonic functions have a power series representation and hence are smooth \cite{a13}. 
  Therefore, since $u-u_0$ is smooth on any $\tilde{K} \Subset K$, any regularity obtained for $u_0$ on $\tilde{K}$ is 
  passed to $u$ on $\tilde{K}$.   
  Initially we only know that $u^+ \in L^2(\Omega)$. From the theory of Reisz potentials, 
  $u_0 \in L^{q}(K)$  for any $1/2 - 1/q < 2s/n$, so $u \in L^q(\tilde{K})$, and so $u^+ \in L^q(\tilde{K})$,
  and we rename $\tilde{K}$ as $K$ and continue the process. Then using again 
  imbeddings of Reisz potentials we obtain that after finitely many iterations (depending on $s$), 
  $u \in L^{\infty}(K)$ and hence $u^+$ as well. Then from \cite{S07} we obtain that 
   \[
    \begin{aligned}
     &u_0 \in C^{0,\alpha}(\R^n) & &\text{ if } \alpha < 2s \leq 1\\
     &u_0 \in C^{1,\alpha}(\R^n) & &\text{ if } \alpha < 2s -1 \text{  and  } 2s>1
    \end{aligned}
   \]
  Again, we obtain the same regularity for $u$ inside $K$. Having obtained initial Holder regularity
  for $u^+$, we obtain higher Holder regularity for $u$ by using Proposition \ref{p:holder}. After finitely 
  many iterations (again depending on $s$) 
  we obtain the conclusion
  of the theorem. Lipschitz regularity is the most we can know for $u^+$, and this is when the iteration process stops.  
 \end{proof}

\section{Topology of the Free Boundary}  \label{s:top}
 For the original local plasma problem studied in \cite{ks78}, the free boundary is exactly 
  \[
   \partial \{u>0\} = \partial \{u<0\} = \{u=0\}.
  \]
 That these two boundaries are the same follows from the maximum and minimum principle since $\Delta u =0$ in 
 $\{u<0\}$ and $\Delta u \leq 0$ in $\{u>0\}$. The situation is different in the nonlocal/fractional case: 
 we cannot apply the same minimum principle to solutions of 
  \[
    \lim_{y \to 0} y^a u_y(x,y) \leq 0
  \]
 since it is possible to have a local minimum and still satisfy the above inequality. It may be possible to construct solutions to \eqref{e:eigen} that satisfy 
  \[
   \partial \{u( \ \cdot \ ,0)<0\}\subsetneq \partial \{u( \ \cdot \ ,0)>0\}.
  \]
 We are mostly interested in the portion of the free boundary $\partial\{u( \ \cdot \ ,0)<0\}$. This next proposition gives an inclusion when we assume
 additionally \eqref{e:do}.

 \begin{proposition} \label{p:free}
   Let $u$ be a solution to \eqref{e:eigen} and assume also \eqref{e:do}, then $\partial \{u<0\} \subset \partial\{u>0\}$. In particular, if
   $u$ is a solution to \eqref{e:p}, then $\partial \{u<\gamma\} \subset \partial\{u>\gamma\}$.
 \end{proposition}

 \begin{proof}
  Let $V$ be an open subset of $\Omega$ such that $V \subset \{u( \ \cdot \ , 0)\leq 0 \}$. Suppose there exists $x_0 \in V$
  such that $(x_0,0) \in \partial \{u(\cdot,0) < 0\}$, and so $u(x_0,0)=0$ and $u$ is not constant. Now 
  $y^a u_y(x,y) = 0$ for $x \in V$. If we evenly reflect in the $y$-variable $u$ is $a$-harmonic on the domain $U \times (-\epsilon,\epsilon)$ for $\epsilon >0$.
  From \eqref{e:do} it follows that $u$ achieves an interior maximum at $(x_0,0)$, and since $u$ is not constant we obtain 
  a contradiction to the maximum principle for $a$-harmonic functions \cite{FKJ}.
 \end{proof}
 
  
  This next Proposition shows that if $u$ is a solution to \eqref{e:p}, then $u$ is strictly subharmonic - in the classical sense for the Laplacian 
  (not fractional Laplacian) on the thin space $\R^n$ -
  across the boundary $\partial \{u>\gamma \}$ which is not 
  true when $s=1$. This illustrates why one cannot hope for a strong minimum principle in $\{u > \gamma\}$. 
  
  \begin{proposition}  \label{p:sub}
   Let $u$ be a solution to \eqref{e:p} with $1/2<s<1$. Then there exists an open set $U$ (in the topology of $\Omega$)
    containing 
   $ \{u(\cdot,0) = \gamma\}$ such that 
    \[
     \Delta u(x,0) > 0 \text{  for every  } x \in U,
    \]
   where $\Delta$ is the classical n-dimensional Laplacian on the thin space $\R^n$.
  \end{proposition} 
  
  \begin{proof}
   Since $s>1/2$, from Theorem \ref{t:solreg} we know $u$ is $C^2$ on the thin space $\R^n$, so $\Delta u(\cdot,0)$ exists in the classical sense. 
   Also, from the extension we have that 
    \[
     v(x,y) = c_a y^a u_y(x,y) \geq 0,
    \]
   where $c_a <0$ is as in \eqref{e:extend},
   and $v$ is $-a$-harmonic. Let $u(x_0,0) = \gamma$. Then 
    \[
     \lim_{y \to 0} y^a u_y(x_0,y) =0. 
    \]
   We then conclude that $(x_0,0)$ is a minimum for the $-a$-harmonic function $v$, and so from Lemma \ref{l:hopf} 
    \[
      \begin{aligned}
       0 &< \lim_{y \to 0} y^{-a} v_y(x_0,y) \\
         &= \frac{1}{c_a}(-\Delta)^{1-s} v(x_0)  \\
         &= \frac{1}{c_a}(-\Delta)^{1-s} (-\Delta)^s u(x_0)  \\
         &= \frac{1}{c_a}(-\Delta) u(x_0).
      \end{aligned}
    \]
   Since $c_a<0$, we conclude that $\Delta u(x,0) >0$ whenever $u(x,0)=0$. From the $C^2$ continuity of $u$ it follows that
   $u$ is strictly subharmonic in a neighborhood of $\{u=\gamma\}$. 
  \end{proof}

  In the next theorem, we show that if 
  $\Omega$ has symmetry in a coordinate direction, then  the same symmetry will
  be inherited in the level sets of solutions of \eqref{e:p} that arise as minimizers of \eqref{e:energy} subject
  to the constraint \eqref{e:constraint}. 
  
  \begin{theorem}  \label{t:symmetry}
   Let $u$ be a solution to \eqref{e:p} that arises as a minimizer of \eqref{e:energy} subject to the constraint
   \eqref{e:constraint}. Assume $\Omega$ is symmetric across a hyperplane $P$ orthogonal to the direction $\nu$.
   Then the level sets of $u$ are also symmetric with respect to $P$.
  \end{theorem} 
  
  \begin{proof}
   We pick our coordinates so that $P=(0,x_2, \ldots , x_{n})$. 
   We use the $a$-harmonic extension. Let $u$ minimize the energy functional
    \[
     \int_{\Omega \times \R}{y^a|\nabla v|^2}
    \]
   subject to the constraint \eqref{e:constraint}. If we Steiner symmetrize \cite{k85} in the $x_1$ direction (which is orthogonal to $y$)
   to obtain $u^*$ , and if $0<T_1 < T_2 < \infty$, then
    \begin{equation}  \label{e:steiner}
     \iint\limits_{\Omega \times [T_1,T_2]}{y^a|\nabla u^*|^2} \leq \iint\limits_{\Omega \times [T_1,T_2]}{y^a|\nabla u|^2}
    \end{equation}
   and equality is achieved if and only if $u$ is already Steiner symmetric in the $x_1$ direction. The case of
   equality holds from the result in \cite{k85} since $a$-harmonic functions are real analytic away from $y=0$. That is why
   initially we restrict ourselves to $T_1>0$. Now we let $T_1 \to 0$ and $T_2 \to \infty$ to obtain
    \[
     \iint\limits_{\Omega \times \R^+}{y^a|\nabla u^*|^2} \leq \iint\limits_{\Omega \times \R^+}{y^a|\nabla u|^2}
    \]
   If $u$ is not already Steiner symmetric in the $x_1$ direction when $y=0$, then by continuity $u$ will not be  
   Steiner symmetric at some time $T>0$, and so by \eqref{e:steiner} the energy of $u^*$ will be less on an interval
   $\Omega \times [T_1 , T_2]$ and hence also on $\Omega \times \R^+$. Notice that the constraint \eqref{e:constraint}
   is preserved for $u^*$. Then if $u$ is a minimizer of \eqref{e:energy}, it must be symmetric in the $x_1$ direction.
  \end{proof}
  
  This next Corollary gives a sufficient condition on the shape of the domain $\Omega$ under which 
  $\partial \{u>\gamma\} = \partial \{u<\gamma\}$ for solutions of \eqref{e:p} that arise as 
  minimizers. A good question would be what conditions are necessary on $\Omega$ in order to assure the same condition on the free boundary.  
  \begin{corollary}  
   Let $u$ be as in Theorem \ref{t:symmetry}, and assume $\Omega$ is symmetric with respect to $x_i$ for
   $1 \leq i \leq n$. Then 
    \[
     \partial \{u( \ \cdot \ ,0)<\gamma\} = \partial \{u( \ \cdot \ ,0)>\gamma\}.
    \]
  \end{corollary}
  
  \begin{proof}
   From Theorem \ref{t:symmetry} $u$ is symmetric in each $x_i$ variable. Thus $u$ achieves a maximum at the origin. 
   If $u(0)\leq \gamma$, then $u$ is $a$-harmonic everywhere and this is a contradiction to the maximum principle. Then $u(0)>\gamma$. 
   Suppose there exists $z=(z_1, z_2, \cdots, z_n, 0) \in \partial \{u(\cdot,0)>\gamma\}$ and $z \notin \partial\{u(\cdot,0)<\gamma\}$. 
   By the symmetry of $u$ we can assume $z_i >0$ for each $i$. Also from the symmetry of $u$ we have  
    \[
     \frac{\partial u}{\partial x_i}(x) \leq 0.
    \]
   For each $x$ in the first quadrant and in particular for the point $z$. If $z \notin \partial \{u<\gamma\}$,
   then by continuity there exists an open region 
    \[
     U:= \{x  \mid x \in B_{\epsilon}(z) \cap \{x_i > z_i\} \ \}
    \]
   and $u(x,0)=\gamma$ for every $x \in U$. $u -\gamma$ is then an $a$-harmonic function with both an even and odd extension (in the $y$-variable) 
   across the thin space. Then from the power series representation of $a$-harmonic functions (see \cite{a13}) it follows that $u$ is constant. This is a contradiction
   since $u(x,y)\to 0$ as $y \to \infty$, and $u(0,0)>\gamma>0$. 
  \end{proof}

\section{Regularity of the Free Boundary}  \label{s:freereg}
 In this section we look at the regularity of the free boundary. Our results apply to solutions of \eqref{e:eigen}.
 We now state a Lemma that will allow us to utilize Almgren's frequency function. For solutions of \eqref{e:eigen}
 Almgren's frequency function will not be monotone. However, we will prove that the limit at the origin exists and use
 this result to put a bound on the dimension of the singular set of free boundary points. We define
  \[
   D(r):= \int_{B_r^+}{y^a |\nabla u|^2} \qquad H(r):= \int_{(\partial B_r)^+}{y^a u^2} 
   \qquad N(r)= r \frac{D(r)}{H(r)}
  \]
 
  \begin{lemma}    \label{l:almgren}
   Let $u$ be a solution to \eqref{e:eigen}. Then $\lim_{r \to 0^+} N(r)$ exists and 
    \[
     0 < \lim_{r \to 0} N(r) < \infty
    \]
  \end{lemma}

  \begin{proof}
   For solutions of \eqref{e:eigen}, $N(r)$ will not necessarily be monotone. We therefore begin by considering the modified function
    \[
     \tilde{N}(r) := r \frac{D(r) - \lambda \int_{B_{r}'}{u^2}}{H(r)}.
    \]
   Both $N(r), \tilde{N}(r)$ are absolutely continuous and hence differentiable for almost every $r$. We note that $\tilde{N}(r)>0$. 
   By the same computations as in \cite{cs07} and using that $u$ is a solution to \eqref{e:eigen} (along with the accompanying $C^{1,\alpha}$ regularity)
   we obtain the following Rellich-type identity
    \begin{equation}  \label{e:d}
     D'(r) = \frac{n-1+a}{r} D(r) + 2 \int_{(\partial B_r)^+}{y^a u_{\nu}^2} + 
             \frac{\lambda}{r} \int_{B_{r}'}{\langle x , \nabla u_+^2 \rangle}.
    \end{equation}
   We also have by routine computations 
    \begin{equation}  \label{e:routine}
     \begin{gathered}
       H'(r) = \frac{n+a}{r} H(r) + 2 \int_{(\partial B_r)^+}{y^a u u_{\nu}} \\
             \frac{d}{dr}\left[ \int_{B_{r}'}{u_+^2} \right] = \int_{\partial B_{r}'}{u_+^2} \\
     \int_{B_{r}'}{\langle x , \nabla u_+^2 \rangle} = 
              r \int_{\partial B_{r}'}{u_+^2} -n \int_{B_{r}'}{u_+^2}.
     \end{gathered}
    \end{equation}
   Combining \eqref{e:d} and \eqref{e:routine} we have 
    \[
     \begin{aligned}
      \frac{d}{dr} \log \tilde{N}(r) &= \frac{1}{r} + \frac{n-1+a}{r}\frac{D(r)}{D(r)-\lambda \int_{B_{r}'}u_+^2} - \frac{n+a}{r} \\
                             & \quad + 2\frac{\int_{(\partial B_r)^+}{y^a u_{\nu}^2}}
                                 {D(r)+ \lambda \int_{B_{r}'}{u_+^2}} 
                              - 2\frac{\int_{(\partial B_r)^+}{y^a u u_{\nu}}}
                                 {\int_{(\partial B_{r})^+}{y^a u^2}} \\ 
                             & \quad + \frac{n\lambda}{r} \frac{\int_{B_{r}'}{u_+^2}}
                                 {D(r)+ \lambda \int_{B_{r}'}{u_+^2}} 
     \end{aligned}
    \]
   Using integration by parts and that $u$ is a solution to \eqref{e:eigen} we have
    \[
     D(r)- \lambda \int_{B_{r}'}{u_+^2} = \int_{(\partial B_r)^+}{y^a u u_{\nu}}.
    \] 
   Then 
    \[
     \begin{aligned}
      \frac{d}{dr} \log \tilde{N}(r) &= \frac{1-a}{r}\left[1 - \frac{D(r)}{D(r)+\lambda \int_{B_{r}'}u_+^2} \right] \\
                                  & \qquad + 2 \left(  \frac{\int_{(\partial B_r)^+}{y^a u_{\nu}^2}}
                                 {\int_{(\partial B_r)^+}{y^a u u_{\nu}}} - 
                                  \frac{\int_{(\partial B_r)^+}{y^a u u_{\nu}}}
                                 {\int_{(\partial B_{r})^+}{y^a u^2}}
                                 \right). 
     \end{aligned}                             
    \]
   The first term is clearly nonnegative, and he second term is nonnegative by the Cauchy-Schwarz inequality. 
   Since $\tilde{N}(r)$ is monotone, 
    the limit as $r \to 0$ exists, and
    \begin{equation}  \label{e:zero}
     0 \leq \lim_{r \to 0^+}\tilde{N}(r) < \infty.
    \end{equation}
   Now from Corollary \ref{c:trace} we have 
   	\[
   	 (1-C\lambda r^{1-a})N(r) \leq \tilde{N}(r) \leq N(r),
   	\]
   and so \eqref{e:zero} holds with $\tilde{N}(r)$ replaced by $N(r)$. To show that the limit is greater than zero we 
   note the following rescaling property $N(r,u)=N(1,u_r)$ where 
    \begin{equation} \label{e:rescale}
     u_r := \frac{u(rx,ry)}{\left(\frac{1}{r^{n+a}}\int_{(\partial B_r)^+}{y^a u^2}\right)^{1/2}} .
    \end{equation}
   From the rescaling and since $N(r)$ is bounded we have 
    \[
     \| u_r \|_{L^2(a, (\partial B_1)^+ )} =1 \qquad \| u_r \|_{H^1(a, B_1^+)} \leq M
    \]
   where $M$ is a constant depending on $u$. Then from Proposition \ref{p:trace} we have that there exists a sequence
   $r_k \to 0$ and a function $u_0$ such that  
    \[
     u_{r_k} \rightharpoonup u_0 \text{  in  } H^1(a, B_1^+) \text{  and  }  
     u_{r_k} \to u_0 \text{  in  } L^2(a, (\partial B_r)^+).
    \]
   We note also that 
    \[
     \int_{B_1^+}{y^a | \nabla u_0|^2} = N(0+).
    \]
   Now 
    \[
     \int_{ (\partial B_r)^+}{y^a u_0^2} = 1,
    \]
   so $u_0$ is not identically zero, and hence we conclude that $N(0+)>0$. 
  \end{proof}
  
  \begin{corollary}  \label{c:homogen}
   Let $u$ be a solution to \eqref{e:eigen}. Let $u_r$ be defined as in \eqref{e:rescale}. Then for 
   every sequence $r_k \to 0$, there exists a subsequence and a function $u_0$ so that 
    \[
      u_{r_k} \rightharpoonup u_0  \text{  in  } H^1(a,B_1^+).
    \] 
   Furthermore, $N(u,0+)=N(u_0,r)$ for every $0<r<1$ and if we evenly reflect $u_0$ in $y$, then
   $u_0$ is $a$-harmonic and homogeneous of degree $N(u,0+)$.  
  \end{corollary}
 
  \begin{proof}
   The weak convergence of a subsequence was established in the proof of Lemma \ref{l:almgren}. Now for
   $t>0$
    \[
     N(t, u_0)= \lim_{r \to 0} N(t,u_r) = \lim_{r \to 0} N(tr,u) = N(u,0+),
    \]
   so $N(r,u_0)\equiv N(0+)$. From the rescaling we have the property 
    \[
     \lim_{y \to 0} y^a \partial_y u_r(x,y) = -r\lambda (u_r)_+. 
    \] 
   If we let $\psi \in C_0^2(B_1)$, then
    \[
     \lim_{ r_{k} \to 0} \int_{B_1^+}{y^a \langle \nabla \psi, \nabla u_{r_k} \rangle} =
      r_k \lambda\int_{B_{1}'}{(u_{r_k})_+ \psi} \to 0.
    \]
   The right hand side goes to zero by utilizing Proposition \ref{p:trace}. Then if we extend $u$ evenly in 
   the $y$ variable, $u$ is $a$-harmonic in $B_1$. 
  \end{proof} 
  
  The blow-up solutions will satisfy \eqref{e:eigen} with $\lambda=0$. We can therefore evenly reflect across the thin space $\{y=0\}$, and each
  blow-up solution will be even and homogeneous.  From \cite{CSS} it follows that the blow-up solutions will be homogeneous of order $k \in \N$. 
  When we have the additional assumption $u_y \leq 0$ - as is the case when $u$ is a solution to \eqref{e:p} - then we have a better classification. 
 
  \begin{corollary}  \label{c:classify}
   Let $u$ be a solution to \eqref{e:eigen} and let $x_0 \in \Gamma$. Assume also $u_y \leq 0$. Let $u_r \to v$ be a blow-up of $u$ at $x_0$. 
   If $\nabla_x u(x_0) \neq 0$, then $v$ is a linear function. If $\nabla_x u(x_0)=0$ we have the following alternative, either $(I)$:
    $v$ is either homogeneous of degree 2
   and of the form 
    \begin{equation}  \label{e:form}
     v(x,y) = p(x) - cy^2
    \end{equation}
   where $p(x)$ is homogeneous of degree 2 in the $x$-variable and $c\geq0$, or $(II)$:
    $v_y \equiv 0$ and $\Delta u(\cdot,0)=0$ and $u$
   is homogeneous of degree $k \in \N$ with $k \geq 2$. 
  \end{corollary}

  \begin{proof}
   When $\nabla_x u(x_0)\neq 0$ this is immediately clear since $v= \nabla_x u(x_0)$. Now since $y^a u_y(x,y) \leq 0$ for all $(x,y)$, 
   this inequality is preserved in the limit for $v$. Then $y^a v_y$ is $-a$-harmonic, nonpositive for $y>0$, and has zero dirichlet data when $y=0$ 
   (since $v$ is even). From the Boundary Harnack Principle \cite{CSS} it follows that $y^a v_y$ 
   is comparable to $y^{1+a}$, or identically zero. If $v_y \equiv 0$ then the classification is immediate. If $v_y$ is not equivalently zero, then
   since $v$ is homogeneous
   $y^a v_y$ is also homogeneous, and thus it is homogeneous of degree $1+a$. Then $v$ is homogeneous of degree 2. Since $v_y\leq 0$ for $y>0$ it follows that
   $v$ must be of the form \eqref{e:form}. (See also \cite{a13} for a classification of homogeneous $a$-harmonic functions). 
  \end{proof}

 We now define the regular set of the free boundary. Let $u$ be a solution to \eqref{e:eigen}.  
  \[
   R_u := \{x \mid u(x,0)=0 \text{  and  } \nabla_x u(x,0)\neq 0 \}.
  \]
 From Corollary \ref{c:homogen} $R_u$ is equivalently the set of free boundary points where \\
 $N(u,x,0+)=1$ or equivalently the set of 
 free boundary points that has a blow-up (and hence every blow-up) that is a linear function independent of the variabley $y$. 
 
 From the implicit function theorem and Theorem \ref{t:solreg} we have the following regularity result for $R(u)$.
 \begin{theorem} \label{t:implicit}
  Let $u$ be a solution to \eqref{e:eigen}. If $u(x_0,0)=0$ and $(x_0,0) \in R_u$, then in a neighborhood 
  $V$ of $x_0$, $\{u(x,0)=0\}$ is a $C^{1,\alpha} \ (C^{2,\alpha})$ graph for $s \leq 1/2 \ (>1/2)$ and $\alpha$
  as in  Theorem \ref{t:solreg}.
 \end{theorem}
 
 It is common in free boundary problems for the free boundary to be more regular than the solution. For instance around regular points 
 for the original local plasma problem \eqref{e:ploc} the free boundary is real analytic \cite{ks78}. A question of interest would be to
 prove higher regularity of the free boundary for solutions of \eqref{e:eigen}.  

  
 


\section{The Singular Set}  \label{s:sing}
  In this section we consider solutions to \eqref{e:eigen} with the additional assumption $u_y \leq 0$. This is the case for instance
  when $u$ is a solution to \eqref{e:p}. The main result of this section is to give a Hausdorff dimensional bound for the singular set 
  defined as 
   \[
    S_u := \{(x,0)\in \partial\{u(\cdot,0)<0\} \mid \nabla_x u(x,0)=0\}.
   \]
  Notice that we do not consider the set $(x,0) \in(\partial\{u(\cdot,0)>0\} \setminus \partial\{u(\cdot,0)<0\})$. 
  The main result of this section is 
   \begin{theorem}  \label{t:singular}
    Let $u$ be a solution to \eqref{e:eigen} in $\Omega$. Then the Hausdorff dimension 
     \[
      \H^\tau(S_u) = 0 
     \]
    for $\tau>n-2$.
   \end{theorem}
  To prove Theorem \ref{t:singular} we follow Federer's method of dimension reduction utilized in \cite{ds15}. The Theorem is an immediate
  consequence of the following two Lemmas. 
  
   \begin{lemma}  \label{l:fed1}
    Assume that $\H^\tau(S_v)=0$ for all functions $v$ which arise as blow-ups of solutions to \eqref{e:eigen}. Then $\H^\tau(S_u)=0$
    for all solutions $u$ to \eqref{e:eigen} in $\Omega$. 
   \end{lemma}  
   
   \begin{proof}
    We first show the following Property $(P)$: for every $x \in S_u$ there exists $d_x >0$ such that for all $\delta \leq d_x$ any subset 
    $D$ of $S_u \cap B_{\delta}(x)$ can be covered by a finite number of balls $B_{r_i}(x_i)$ with $x_i \in D$ such that
     \[
      \sum_{i} r_i^\tau \leq \frac{\delta^\tau}{2}.
     \]
    We show $(P)$ using a compactness argument. Suppose by way of contradiction the conclusion is not true for a sequence $\delta_k \to 0$. By picking a
    subsequence if necessary we perform a blow-up at $x$, $u_{\delta_k} \to u_0$. By assumption there exists finitely many $B_{r_i/4}(x_i)$ such that
     \[
      S_{u_0} \subset \cup B_{r_i/4} \quad \text{  and  } \quad \sum_{i} r_i^\tau \leq \frac{1}{2}. 
     \]
    Since $u_{\delta_k} \to u_0 $ in $C^{1,\beta}$ it follows that there exists $k_0$ such that if $k>k_0$ then
     \[
      S_{u_{\delta_k}} \cap B_1 \subset \cup B_{r_i/4}.
     \]
    Then rescaling backward $u$ satisfies the conclusion for all large $k$ in $B_{\delta_k}$ and we reach a contradiction. 
    
    We now denote $D_k:= \{ y \in S_u \mid d_x \geq 1/k\}$. Fix $x_o \in D_k$. By Property $(P)$ we can cover $D_k \cap B_{r_0}(x_0)$
    where $r_0 = 1/k$ with a finite number of balls $B_{r_i}(x_i)$ with $x_i \in D_k$ and
     \[
      \sum_{i} r_i^\tau \leq \frac{r_o^\tau}{2}.
     \]
    Now we repeat the same argument for each $B_{r_i}(x_i)$ and cover it with balls $B_{r_{ij}}(x_{ij})$ to obtain 
     \[
      \sum_{j} r_{ij}^\tau \leq \frac{r_i^\tau}{2}.
     \]
    Repeating this argument $m$ times we obtain $\H^s (D_k \cap B_{r_0}(x_0))=0 $. We then let $k \to \infty$ to conclude the Lemma.
    \end{proof}
   
   \begin{lemma}  \label{l:fed2}
    Let $u_0$ be a blow-up of $u_k(x_0)$ where $u$ is a solution to \eqref{e:eigen} and $x_o \in S_{u}$. Then 
     \[
      \H^\tau (S_{u_0}) =0 \text{  for  } \tau>n-2.
     \]
   \end{lemma}
   
   \begin{proof}
    We begin by considering the problem in dimension $n=1$, and suppose there exists $x_0 \in S_u$. From Corollary \ref{c:classify} $u_0=ax^2 -cy^2$. 
    Since $x_0 \in \partial \{u(x,0)<0\}$, it follows that $0 \in \partial\{u_0(x,0)<0\}$. Since $c\leq 0$ we have $a \geq 0$ 
    and we immediately obtain a contradiction. The result now follows by using the standard argument of Federer for homogeneous solutions. 
    In dimension $n=2$ if $0 \neq x_0 \in S_{u_0}$ then performing a blow-up at $x_0$ and using the homogeneity of $u_0$ we obtain a blow-up limit $u_{00}$ in $n=1$
    with $0 \in S_{u_{00}}$ which is a contradiction. The argument for higher dimensions then follows using the ideas of Lemma \ref{l:fed1} as in \cite{ds15}. 
   \end{proof}

   We remark here that using the same arguments it is easy to show $\H^{\tau}(\Gamma)=0$ for $\tau>n-1$, so that the free boundary is of Hausdorff dimension $n-1$. 
   

\bibliographystyle{amsplain}
\bibliography{refplasma}

\providecommand{\bysame}{\leavevmode\hbox to3em{\hrulefill}\thinspace}
\providecommand{\MR}{\relax\ifhmode\unskip\space\fi MR }
\providecommand{\MRhref}[2]{%
  \href{http://www.ams.org/mathscinet-getitem?mr=#1}{#2}
}
\providecommand{\href}[2]{#2}
\begin{thebibliography}{10}

\bibitem{a13}
Mark Allen, \emph{Thin free boundary problems}, Ph.D. thesis, Purdue
  University, 2013.

\bibitem{ALP}
Mark Allen, Erik Lindgren, and Arshak Petrosyan, \emph{The two-phase fractional
  obstacle problem}, SIAM J. Math. Anal. \textbf{47} (2015), no.~3, 1879--1905.
  \MR{3348118}

\bibitem{bb80}
Henri Berestycki and Ha{\"{\i}}m Br{\'e}zis, \emph{On a free boundary problem
  arising in plasma physics}, Nonlinear Anal. \textbf{4} (1980), no.~3,
  415--436. \MR{574364 (83b:76096)}

\bibitem{cs07}
Luis Caffarelli and Luis Silvestre, \emph{An extension problem related to the
  fractional {L}aplacian}, Comm. Partial Differential Equations \textbf{32}
  (2007), no.~7-9, 1245--1260. \MR{2354493 (2009k:35096)}

\bibitem{crs10}
Luis~A. Caffarelli, Jean-Michel Roquejoffre, and Yannick Sire,
  \emph{Variational problems for free boundaries for the fractional
  {L}aplacian}, J. Eur. Math. Soc. (JEMS) \textbf{12} (2010), no.~5,
  1151--1179. \MR{2677613 (2011f:49024)}

\bibitem{CSS}
Luis~A. Caffarelli, Sandro Salsa, and Luis Silvestre, \emph{Regularity
  estimates for the solution and the free boundary of the obstacle problem for
  the fractional {L}aplacian}, Invent. Math. \textbf{171} (2008), no.~2,
  425--461. \MR{2367025 (2009g:35347)}

\bibitem{ds15}
Daniela De~Silva and Ovidiu Savin, \emph{Regularity of {L}ipschitz free
  boundaries for the thin one-phase problem}, J. Eur. Math. Soc. (JEMS)
  \textbf{17} (2015), no.~6, 1293--1326. \MR{3353802}

\bibitem{FKJ}
E.~B. Fabes, C.~E. Kenig, and D.~Jerison, \emph{Boundary behavior of solutions
  to degenerate elliptic equations}, Conference on harmonic analysis in honor
  of {A}ntoni {Z}ygmund, {V}ol. {I}, {II} ({C}hicago, {I}ll., 1981), Wadsworth
  Math. Ser., Wadsworth, Belmont, CA, 1983, pp.~577--589. \MR{730093
  (85m:35028)}

\bibitem{k85}
Bernhard Kawohl, \emph{Rearrangements and convexity of level sets in {PDE}},
  Lecture Notes in Mathematics, vol. 1150, Springer-Verlag, Berlin, 1985.
  \MR{810619 (87a:35001)}

\bibitem{ks78}
David Kinderlehrer and Joel Spruck, \emph{The shape and smoothness of stable
  plasma configurations}, Ann. Scuola Norm. Sup. Pisa Cl. Sci. (4) \textbf{5}
  (1978), no.~1, 131--148. \MR{0481511 (58 \#1627)}

\bibitem{S07}
Luis Silvestre, \emph{Regularity of the obstacle problem for a fractional power
  of the {L}aplace operator}, Comm. Pure Appl. Math. \textbf{60} (2007), no.~1,
  67--112. \MR{2270163 (2008a:35041)}

\bibitem{st10}
Pablo~Ra{\'u}l Stinga and Jos{\'e}~Luis Torrea, \emph{Extension problem and
  {H}arnack's inequality for some fractional operators}, Comm. Partial
  Differential Equations \textbf{35} (2010), no.~11, 2092--2122. \MR{2754080
  (2012c:35456)}

\bibitem{t75}
R.~Temam, \emph{A non-linear eigenvalue problem: the shape at equilibrium of a
  confined plasma}, Arch. Rational Mech. Anal. \textbf{60} (1975/76), no.~1,
  51--73. \MR{0412637 (54 \#759)}

\bibitem{t77}
\bysame, \emph{Remarks on a free boundary value problem arising in plasma
  physics}, Comm. Partial Differential Equations \textbf{2} (1977), no.~6,
  563--585. \MR{0602544 (58 \#29213)}

\bibitem{z95}
Eberhard Zeidler, \emph{Applied functional analysis}, Applied Mathematical
  Sciences, vol. 109, Springer-Verlag, New York, 1995, Main principles and
  their applications. \MR{1347692 (96i:00006)}

\end{thebibliography}

\end{document}